\documentclass{amsart}
\usepackage{amsmath}
\allowdisplaybreaks[4]

\theoremstyle{plain}
\newtheorem{theorem}{Theorem}
\newtheorem{corollary}[theorem]{Corollary}
\newtheorem{lemma}[theorem]{Lemma}
\newtheorem{proposition}[theorem]{Proposition}
\theoremstyle{definition}
\newtheorem*{example}{Example}
\theoremstyle{remark}
\newtheorem*{remark*}{Remark}

\newcommand\CC{{\mathbb C}}
\newcommand\RR{{\mathbb R}}
\newcommand\DD{{\mathbb D}}
\newcommand\NN{{\mathbb N}}
\newcommand\bS{{\mathbb S}}
\newcommand\TT{{\mathbb T}}
\newcommand\BB{{\mathbb B}}
\newcommand\MM{{\mathbb M}}
\newcommand\HH{{\mathbb H}}
\newcommand\pM{{\partial\MM}}
\newcommand\intH{\int_\HH}
\newcommand\intM{\int_\MM}
\newcommand\intpM{\int_\pM}
\renewcommand\[{\begin{equation}}
\renewcommand\]{\end{equation}}
\newcommand\spr[1]{\langle#1\rangle}
\newcommand\bs{{\mathbf s}}
\newcommand\cS{\mathcal S}
\newcommand\oy{\overline y}
\newcommand\oz{\overline z}
\newcommand\ow{\overline w}
\newcommand\oeta{\overline\eta}
\newcommand\PP{\mathcal P}
\newcommand\dbar{\overline\partial{}}

\newcommand\cb{\bullet}
\newcommand\OnC{O(n+1,\CC)}
\newcommand\OnR{O(n+1,\RR)}
\newcommand\rtm{{\rho\otimes\mu}}
\newcommand\drtm{d(\rtm)}
\newcommand\suml{\sum_{l=0}^\infty}
\newcommand\sumk{\sum_{k=0}^\infty}
\newcommand\sumj{\sum_{j=0}^\infty}
\newcommand\ddlh{\partial\dbar\log\frac1h}
\newcommand\Ho[1]{H_{\overline{#1}}}
\newcommand\cH{\mathcal H}

\begin{document}

\title[Kepler manifold]{Bergman kernels, TYZ expansions and Hankel operators
 on the Kepler manifold}
\author{H\'el\`ene Bommier-Hato}
\address{Aix-Marseille Universit\'e, I2M UMR CNRS~7373,
 39~Rue F.~Joliot-Curie, 13453~Marseille Cedex~13, France}
\email{helene.bommier@gmail.com}
\author{Miroslav~Engli\v s}
\address{Mathematics Institute, Silesian University in Opava,
 Na~Rybn\'\i\v cku~1, 74601~Opava, Czech Republic {\rm and}
 Mathematics Institute, \v Zitn\' a 25, 11567~Prague~1, Czech Republic}
\email{englis{@}math.cas.cz}
\thanks{Research supported by GA\v CR grant no.~201/12/0426
 %, GA~AV~\v CR grant no.~IAA100190802
 %  Barrande project MEB021108
 and RVO funding for I\v CO~67985840.}
\author{El-Hassan Youssfi}
\address{Aix-Marseille Universit\'e, I2M UMR CNRS~7373,
 39 Rue F-Juliot-Curie, 13453~Marseille Cedex 13, France}
\email{el-hassan.youssfi@univ-amu.fr}
\begin{abstract}
For~a class of $\OnR$ invariant measures on the Kepler manifold possessing
finite moments of all orders, we~describe the reproducing kernels of the
associated Bergman spaces, discuss the corresponding asymptotic expansions
of Tian-Yau-Zelditch, and study the relevant Hankel operators with conjugate
holomorphic symbols.
Related reproducing kernels on the minimal ball are also discussed.
Finally, we~observe that the Kepler manifold either does
not admit balanced metrics, or such metrics are not unique.
\end{abstract}
\maketitle

\section{Introduction}
Let $n\ge2$ and consider the Kepler manifold in $\CC^{n+1}$ defined~by
$$ \HH := \{z\in\CC^{n+1}: z\cb z=0, \; z\neq0 \}, $$
where $z\cb w:=z_1 w_1+\dots+z_{n+1}w_{n+1}$. This is the orbit of the vector
$e=(1,i,0,\dots,0)$ under the $\OnC$-action on $\CC^{n+1}$; it~is also
well-known that $\HH$ can be identified with the cotangent bundle of the unit
sphere $\bS^n$ in $\RR^{n+1}$ minus its zero section. The~unit ball of~$\HH$,
$$ \MM := \{ z\in\HH: |z|^2 = z\cb\oz <1 \}  $$
as~well as its boundary $\pM=\{z\in\HH:|z|=1\}$ are invariant under
$\OnC$ $\cap U(n+1)=\OnR$, and in fact $\pM$ is the orbit of $e$ under $\OnR$.
In~particular, there is a unique $\OnR$-invariant probability measure $d\mu$
on~$\pM$, coming from the Haar measure on the (compact) group $\OnR$.
Explicitly, denoting
$$ \alpha := (n+1) \frac{(-1)^{j-1}}{z_j}\,dz_1\wedge\dots\wedge\widehat{dz_j}
 \wedge\dots\wedge dz_{n+1} \qquad\text{on } z_j\neq0   $$
(this~is, up~to constant factor, the unique $S\OnC$-invariant holomorphic
$n$-form on~$\HH$, see~\cite{OPY}) and defining a $(2n-1)$-form $\omega$ on
$\pM$~by 
$$ \omega(z)(V_1,\dots,V_{2n-1}) := \alpha(z)\wedge\overline{\alpha(z)}
 (z,V_1,\dots,V_{2n-1}), \qquad V_1,\dots,V_{2n-1}\in T(\pM),  $$
we~then have $d\mu=\omega/\omega(\pM)$ (where, abusing notation, we~denote
by $\omega$ also the measure induced by $\omega$ on~$\pM$).
It~follows, in~particular, that $d\mu$ is also invariant under complex
rotations 
$$ z\mapsto\epsilon z, \qquad \epsilon\in\TT=\{z\in\CC:|z|=1\}.   $$
For~a~finite (nonnegative Borel) measure $d\rho$ on $(0,\infty)$,
we~can therefore define a rotation-invariant measure $\drtm$ on $\HH$~by
$$ \intH f \,\drtm := \int_0^\infty \intpM f(t\zeta) \, d\mu(\zeta)
 \, d\rho(t),  $$
and the (weighted) Bergman space
$$ A^2_\rtm := \{ f\in L^2(\drtm): f\text{ is holomorphic on }R\MM \},  $$
where $R:=\sup\{t>0:t\in\operatorname{supp}\rho\}=\sup\{|z|:
z\in\operatorname{supp}\rtm\}$ (with the understanding that
$R\MM=\HH$ if $R=+\infty$). It~is standard that $A^2_\rtm$ is a reproducing
kernel Hilbert space, that~is, there exists a function $K_\rtm(x,y)\equiv
K(x,y)$ on $R\MM\times R\MM$ for which $K(\cdot,y)\in A^2_\rtm$ for each~$y$,
$K(y,x)=\overline{K(x,y)}$, and
$$ f(z) = \int_{R\MM} f(w) K(z,w) \, \drtm(w) \qquad \forall f\in A^2_\rtm.  $$

Our~goal in this paper is to give a description of these reproducing kernels,
establish their asymptotics as $\rho$ varies in a certain~way
(so-called Tian-Yau-Zelditch, or~TYZ, expansion), and study the Hankel
operators on~$A^2_\rtm$. We~also give an analogous description for the
reproducing kernels on the minimal ball
$$ \BB := \{ z\in\CC^n: |z|^2+|z\cb z|<1 \} ,  $$
which is the image of $\MM$ under the 2-sheeted proper holomorphic mapping
given by the projection onto the first $n$ coordinates.

In~more detail, let
$$ q_k := \int_0^\infty t^k \, d\rho(t)   $$
be the moments of the measure $d\rho$ (the~values $q_k=+\infty$ being also
allowed if the integral diverges). Our~starting point is the following formula
for the reproducing kernels (whose proof goes by arguments which are already
quite standard).

\begin{theorem} \label{pPA} For $z,w\in R\MM$ with $R$ as above,
$$ K_\rtm(z,w) = \suml \frac{(z\cb\ow)^l}{d_l},   $$
with
\[ d_l := \frac{q_{2l}}{N(l)}  \label{tDL}    \]
where
$$ N(l) := \binom{l+n-1}{n-1} + \binom{l+n-2}{n-1} 
 = \frac{(2l+n-1)(l+n-2)!}{l!(n-1)!}.   $$
\end{theorem}

Next, recall that, quite generally, for an $n$-dimensional complex manifold $M$
and a holomorphic line bundle $L$ over $M$ equipped with a Hermitian metric,
the so-called Kempf distortion functions $\epsilon_l$, $l=0,1,2,\dots$,
are~defined~by
$$ \epsilon_l(z) := \sum_j h(s_j(z),s_j(z)),  $$
where $\{s_j\}_j$ is an orthonormal basis of the Hilbert space
$L^2_{\text{hol}}(L^{\otimes l},\omega^n)$ of holomorphic sections of the
$l$-th tensor power $L^{\otimes l}$ of~$L$ square-integrable with respect
to the volume element $\omega^n$ on~$M$, where $\omega=-\operatorname{curv}h$
(which is assumed to be positive); see~Kempf~\cite{Kem}, Rawnsley~\cite{Rawn},
Ji~\cite{Ji} and Zhang~\cite{Zha}. These functions are of importance in the
study of projective embeddings and constant scalar curvature metrics
(Donaldson~\cite{Don}), where a prominent role is played, in~particular,
by~their asymptotic behaviour as $l$ tends to infinity: namely, one~has
$$ \epsilon_l(z) \approx l^n \sumj a_j(z) l^{-j} \qquad\text{as }l\to+\infty $$
in~the $C^\infty$-sense, with some smooth coefficient functions~$a_j(z)$,
and $a_0(z)=1$. This has been established in various contexts by
Berezin~\cite{Ber} (for~bounded symmetric domains), Tian~\cite{Tian} and
Ruan~\cite{Ruan} (answering a conjecture of~Yau) and Catlin~\cite{Catl} and
Zelditch~\cite{Zeld} for $M$ compact, Engli\v s~\cite{E32} (for~$M$ a strictly
pseudoconvex domain in $\CC^n$ with smooth boundary and $h$ subject to some
technical hypotheses), etc. If~$M$ is a domain in $\CC^n$ and the line bundle
is trivial (which certainly happens if $M$ is simply connected; in~this case
one need not restrict to integer~$l$, but may allow it to be any positive
number), one~can identify (holomorphic) sections of $L$ with (holomorphic)
functions on~$M$, $h$~with a positive smooth weight on~$M$,
$L^2_{\text{hol}}(L^{\otimes l},\omega^n)$ with $A^2_{h^l\omega^n}=
A^2_{h^l\det[\ddlh]}$, and
\[ \epsilon_l(z) = h(z)^l K_{h^l\det[\ddlh]}(z,z),   \label{tDE}  \]
where we (momentarily) denote by $K_w$ the weighted Bergman kernel with respect
to a weight $w$ on~$M$ (and~similarly for~$A^2_w$).
In~the context of our Kepler manifold, this has been studied by Gramchev
and~Loi \cite{GrL} for $L=\HH\times\CC$ (i.e.~the~trivial bundle) and
$h(z)=e^{-|z|}$ (so~$\omega=\frac i2\partial\dbar|z|$; this turns out to be
the symplectic form inherited from the isomorphism $\HH\cong T^*\bS^n\setminus
\{\text{zero section}\}$ mentioned in the beginning of this paper~\cite{Rawn}),
who~showed that
\[ \epsilon_l(z) = l^n + \frac{(n-2)(n-1)}{2|z|} l^{n-1} +
 \sum_{k=2}^{n-2} \frac{b_k}{|z|^k} l^{n-k} + R_l(z),   \label{tDY}   \]
with some constants $b_k$ independent of $z$ and remainder term
$R_l(z)=O(e^{-cl|z|})$ for some $c>0$ (i.e.~exponentially small).

Our~second main result is the following.

\begin{theorem} \label{pPB}
Let $K_s=K_\rtm$ for
\[ d\rho(t) = 2cm e^{-st^{2m}} t^{2mn-1} \, dt,  \label{tDR}  \]
where $c,m$ are fixed positive constants and $s>0$. Then as $s\to+\infty$,
\[ e^{-s|z|^{2m}} K_s(z,z) = \frac{2m^n}{(n-1)!c} s^n \sum_{j=0}^{n-1}
 \frac{b_j}{s^j|z|^{2mj}} + R_s(z),  \label{tDQ}   \]
where $b_j$ are constants depending on $m$ and $n$ only,
$$ b_0=1, \quad b_1=\frac{(1-n)(mn-n+1)}{2m} ,  $$
and $R_s(z)=O(e^{-\delta s|z|^{2m}})$ with some $\delta>0$.
\end{theorem}

The~result of Gramchev and Loi corresponds to $m=\frac12$ and $c=\frac{2^{1-n}}
{(n-1)!}$, so~it is recovered as a special case. Our~method of proof is a good
deal simpler than in~\cite{GrL}, covers all $m>0$, and is also extendible to
other situations. We~further note that $\drtm$ with the $\rho$ from~\eqref{tDR}
actually coincides (up~to a constant factor) with $\omega^n$ for $\omega=\frac
i2 \partial\dbar(s|z|^{2m})$, in~full accordance with~\eqref{tDE}
(taking~$h(z)=e^{-|z|^{2m}}$, and with $l=1,2,\dots$ replaced by the continous
parameter $s>0$ as already remarked above). 

There is also an analogous result for the Bergman-type weights on~$\MM$
corresponding to $d\rho(t)=\chi_{[0,1]}(t)(1-t^2)^s t^{2n-1}\,dt$, $s>-1$. 

As~for our last topic, recall that the Hankel operator~$\Ho g$,
$g\in A^2_\rtm$, is~the operator from $A^2_\rtm$ into $L^2(\rtm)$ defined~by
$$ \Ho g f := (I-P)(\overline g f),   $$
where $P:L^2(\rtm)\to A^2_\rtm$ is the orthogonal projection.
This is a densely defined operator, which~is (extends~to~be)
bounded e.g.~whenever $g$ is bounded.
For~the (analogously defined) Hankel operators on the unit disc $\DD$ in~$\CC$
or the unit ball $\BB^n$ of~$\CC^n$, $n\ge2$, criteria for the membership of
$\Ho g$ in the Schatten classes~$\cS^p$, $p>0$, were given in the classical
papers by Arazy, Fisher and Peetre \cite{AFPD}~\cite{AFPB}: it~turns out that
for $p\le1$ there are no nonzero $\Ho g$ in $\cS^p$ on~$\DD$, while for $p>1$,
$\Ho g\in\cS^p$ if~and only~if $g\in B^p(\DD)$, the $p$-th order Besov space
on~$\DD$; while on~$\BB^n$, $n\ge2$, there are no nonzero $\Ho g$ in $\cS^p$
if $p\le 2n$, while for $p>2n$, again $\Ho g\in\cS^p$ if~and only~if
$g\in B^p(\BB^n)$. One~says that there is a cut-off at $p=1$ or $p=2n$,
respectively. The~result remains in force also for $\DD$ and $\BB^n$ replaced
any bounded strictly-pseudoconvex domain in $\CC^n$, $n\ge1$, with smooth
boundary (Luecking~\cite{Lue}). 

Our~third main result shows that for the Bergman space
$A^2(\MM):=A^2_\rtm$ for $d\rho(t)=\chi_{[0,1]}(t) t^{2n-1}dt$ on~$\MM$,
there is also a cut-off at $p=2n$.

\begin{theorem} \label{pPC}
Let $p\ge1$. Then the following are equivalent.
\begin{itemize}
\item[(i)] There exists nonconstant $g\in A^2(\MM)$ with $\Ho g\in\cS^p$.
\item[(ii)] There exists a nonzero homogeneous polynomial $g$ of degree $m\ge1$
such that $\Ho g\in\cS^p$.
\item[(iii)] There exists $m\ge1$ such that $\Ho g\in\cS^p$ for all homogeneous
polynomials $g$ of degree~$m$.
\item[(iv)] $p>2n$.
\item[(v)] $\Ho g\in\cS^p$ for any polynomial~$g$.
\end{itemize}
\end{theorem}

We~remark that the proofs in \cite{AFPD} and \cite{AFPB} relied on the
homogeneity of $\BB^n$ under biholomorphic self-maps, and thus are not
directly applicable for $\MM$ with its much smaller automorphism group.
The~proof in \cite{Lue} relied on $\dbar$-techniques, which probably 
could be adapted to our case of~$\MM$ i.e.~of a smoothly bounded strictly
pseudoconvex domain not in $\CC^n$ but in a complex manifold, and 
furthermore having a singularity in the interior (cf.~Ruppenthal~\cite{Rup}).
Our~method of proof of Theorem~\ref{pPC}, which is close in spirit to those
of \cite{AFPD} and~\cite{AFPB}, is, however, much more elementary.

The~paper is organized as follows. The~proof of Theorem~\ref{pPA},
together with miscellaneous necessary prerequisites and the results
for the minimal~ball, occupies Section~2. Applications to the TYZ expansion
appear in Section~3, and the results on Hankel operators in Section~4.
The~final section, Section~5, concludes by a small observation concerning
the so-called balanced metrics on~$\HH$ in the sense of Donaldson~\cite{Don}.

Large part of this work was done while the second named author was visiting
the other two; the~hospitality of Universite Aix Marseille in this
connection is gratefully acknowledged.

\section{Reproducing kernels}
Denote by $\PP_k$, $k=0,1,2,\dots$, the~space of (restrictions to~$\HH$~of)
polynomials on $\CC^{n+1}$ homogeneous of degree~$k$. Clearly, $\PP_k$~is
isomorphic to the quotient of the analogous space of $k$-homogeneous
polynomials on all of~$\CC^n$ by the $k$-homogeneous component of the ideal
generated by $z\cb z$, and
$$ \dim \PP_k = \binom{k+n-1}{n-1}+\binom{k+n-2}{n-1} =: N(k).   $$
For $z\in\MM$ and $f\in L^2(\bS^n,d\sigma)$, where $\sigma$ stands for the
normalized surface measure on~$\bS^n$, set
\[ \hat f(z) := \int_{\bS^n} f(\zeta) e^{\spr{z,\zeta}}\,d\sigma(\zeta),
 \label{tDW}  \]
and let $\cH^k\equiv\cH^k(\bS^n)$ denote the subspace in $L^2(\bS^n,d\sigma)$
consisting of spherical harmonics of degree~$k$. It~is then known \cite{Ii}
\cite{Wada} that the functions $x\mapsto(x\cb z)^k$, $z\in\HH$, span~$\cH^k$,
the~functions $z\mapsto(z\cb\zeta)^k$, $\zeta\in\bS^n$, span~$\PP_k$, and the
mapping $f\mapsto\hat f$ is an isomorphism of $\cH^k$ onto~$\PP_k$,
for~each~$k$. Using the Funke-Hecke theorem \cite[p.~20]{Mull},
it~then follows that \cite{OPY}~\cite{MY}
\[ \intpM (z\cb w)^k (\xi\cb\ow)^l \,d\mu(w)
 = \delta_{kl} \frac{(z\cb\xi)^k}{N(k)}   \label{tTA}   \]
for all $z\in\HH$, $\xi\in\CC^{n+1}$; and, consequently, $\PP_k$~and $\PP_l$
are orthogonal in $L^2(d\mu)$ if $k\neq l$, while
\[ \intpM f(w) (z\cb\ow)^k \,d\mu(w) = \frac{f(z)}{N(k)}  \label{tTB}  \]
for all $z\in\HH$ and $f\in\PP_k$.

If~$f$ is a function holomorphic on $R\MM$ (for~some $0<R\le\infty$),
then it has a unique decomposition of the form
\[ f = \sumk f_k, \qquad f_k\in\PP_k ,   \label{tDC}   \]
with the sum converging absolutely and uniformly on compact subsets of~$R\MM$
(\cite[Lemma~3.1]{MY}). Let~$\bs=(s_0,s_1,\dots)$ be an arbitrary sequence of
positive numbers. We~denote by $A^2_\bs$ the space of all functions $f$
holomorphic in some $R\MM$, $R>0$, for~which,
$$ \|f\|_\bs^2 := \sum s_k \|f_k\|^2_{L^2(\pM,d\mu)} < \infty,   $$
equipped with the natural inner product
$$ \spr{f,g}_\bs := \sumk s_k \spr{f_k,g_k}_{L^2(\pM,d\mu)}   $$
for $f=\sum_k f_k$, $g=\sum_k g_k$ as in~\eqref{tDC}. It~is immediate that
$A^2_\bs$ is a Hilbert space which contains each~$\PP_k$, and the linear
span of the latter (i.e.~the~space of all polynomials on~$\HH$) is~dense
in~$A^2_\bs$. 

\begin{proposition} \label{pPD}
Assume that
$$ R_\bs := \liminf_{k\to\infty} \Big|\frac{N(k)}{s_k}\Big|^{-1/2k} >0  $$
$($the~value $R_\bs=\infty$ being also allowed$)$.
Then $A^2_\bs$ is a reproducing kernel Hilbert space of holomorphic functions
on~$R_\bs\MM$, with reproducing kernel
\[ K_\bs(x,y) = \sumk \frac{N(k)}{s_k} (x\cb\oy)^k .  \label{tTC}   \]
\end{proposition}

\begin{proof}
By~the definition of~$R_\bs$, the~series \eqref{tTC} converges pointwise and
locally uniformly for $x,y\in R_\bs\MM$. Moreover, for $g=K_\bs(\cdot,y)$
we~plainly have $g_k=\frac{N(k)}{s_k}(\cdot\cb\oy)^k$ and by~\eqref{tTA},
$\|g_k\|^2_{L^2(\pM,d\mu)}=\frac{N(k)}{s_k^2} (y\cb\oy)^k$,~so
$$ \|g\|_\bs^2 = \sumk \frac{N(k)}{s_k}(y\cb\oy)^k = K_\bs(y,y) < \infty.  $$
Thus $K_\bs(\cdot,y)\in A^2_\bs$ for $y\in R_\bs\MM$.
Furthermore, for $f\in A^2_\bs$, by~\eqref{tTB}
\begin{align*}
\sum_k |f_k(y)|
&\le \sum_k \Big| N(k) \intpM f_k(w)(y\cb\ow)^k \, d\mu(w) \Big| \\
&= \sum_k s_k \Big| \Big\langle f_k,\frac{N(k)}{s_k}(\cdot\cb\oy)^k
 \Big\rangle _{L^2(\pM,d\mu)} \Big|   \\
&\le \sum_k s_k \|f_k\|_{L^2(\pM,d\mu)}
 \Big\| \frac{N(k)}{s_k}(\cdot\cb\oy)^k \Big\|_{L^2(\pM,d\mu)}   \\
&\le \Big( \sum_k s_k \|f_k\|^2_{L^2(\pM,d\mu)} \Big)^{1/2}
 \Big( \sum_k \Big\| \frac{N(k)}{s_k}(\cdot\cb\oy)^k \Big\|_{L^2(\pM,d\mu)}^2
 \Big)^{1/2}   \\
&= \|f\|_\bs K_\bs(y,y)^{1/2} , 
\end{align*}
implying that the series $\sum_k f_k$ converges locally uniformly on $R_\bs\MM$
and that the point evaluation $f\mapsto f(y)$ is continuous on $A^2_\bs$ for
all $y\in R_\bs\MM$. Finally, removing absolute values in the last computation
shows that $f(y)=\spr{f,K_\bs(\cdot,y)}_\bs$, proving that $K_\bs$ is the
reproducing kernel for~$A^2_\bs$.
\end{proof}

Recall now that a sequence $\bs=(s_k)_{k\in\NN}$ is called a Stieltjes moment
sequence if it has the form
$$ s_k = s_k(\nu) := \int_0^\infty r^k \, d\nu(r)   $$   %(
for some nonnegative measure $\nu$ on $[0,+\infty)$,    %]
called a representing measure for~$\bs$; or,~alternatively,
$$ s_k = \int_0^\infty t^{2k} \, d\rho(t)   $$
for the nonnegative measure $d\rho(t)=d\nu(t^2)$. These sequences have been
characterized by Stieltjes in terms of a positive definiteness conditions.
It~follows from the above integral representation that each Stieltjes moment
sequence is either non-vanishing, that~is, $s_k>0$ for all~$k$, or~else
$s_k =c\delta_{0k}$ for all $k$ for some $c\ge0$.
Fix~a nonvanishing Stieltjes moment sequence~$\bs=(s_k)$.
By~the Cauchy-Schwarz inequality we see that the sequence $\frac{s_{k+1}}{s_k}$
is~nondecreasing and hence converges as $k\to+\infty$ to the radius of
convergence $R_\bs^2$ of the series  
$$ \sumk \frac{N(k)}{s_k} z^k,   \qquad z\in\CC.  $$ 
We~can now prove our first theorem from the Introduction.

\newtheorem*{proofA}{Proof of Theorem~\ref{pPA}}
\begin{proofA}
Recall that we denoted $R=\sup\{t>0:t\in\operatorname{supp}\rho\}$.
If~$R=\infty$, then for any $r>0$ we have $q_{2k}\ge\int_r^\infty t^{2k}
\,d\rho(t) \ge r^{2k} \rho((r,\infty))$ with $\rho((r,\infty))>0$,
so $\liminf_{k\to\infty} q_{2k}^{1/2k}\ge r$; thus $q_{2k}^{1/2k}\to+\infty$.
If~$R<\infty$, then the same argument shows that $\liminf_{k\to\infty}
q_{2k}^{1/2k}\ge r$ for any $r<R$, while from
$$ q_{2k} \le R^{2k} \int_0^R \,d\rho = R^{2k} \rho([0,R])  $$
with $\rho([0,R])<\infty$ we get $\limsup_{k\to\infty} q_{2k}^{1/2k}\le R$.
Setting $s_k=q_{2k}$ we~thus get in either case
$$ R=R_\bs.  $$
Now for any $r<R$ and $f\in A^2_\rtm$, we~have by the uniform convergence
of~\eqref{tDC} 
\begin{align*}
\int_{r\MM} |f|^2 \,\drtm
&= \sum_{j,k} \int_{r\MM} f_j \overline{f_k} \,\drtm  \\
&= \sum_{j,k} \int_0^r \intpM t^{j+k} f_j(\zeta) \overline{f_k(\zeta)}
 \, d\mu(\zeta) \, d\rho(t)  \\
&= \sum_k \Big( \int_0^r t^{2k} \,d\rho(t) \Big) \|f_k\|^2_{L^2(\pM,d\mu)}
\end{align*}
by~the orthogonality of $\PP_k$ and $\PP_l$ for $j\neq k$.
Letting $r\nearrow R$, we~thus~get
$$ \|f\|^2_\rtm = \sum_k q_{2k} \|f_k\|^2_{L^2(\pM,d\mu)} = \|f\|^2_\bs.  $$
Hence $A^2_\rtm\subset A^2_\bs$, with equal norms. Since clearly $\PP_l\in
A^2_\rtm$ for each $l$ and the span of $\PP_l$ is dense in~$A^2_\bs$,
it~follows that actually $A^2_\rtm=A^2_\bs$ and $\|f\|_\rtm=\|f\|_\bs$
for any $f\in A^2_\rtm$. The~claim now follows from the last proposition.
\qed \end{proofA}

As~an example, let~$\phi$ be a nonnegative integrable function on~$(0,\infty)$,
and consider the volume element
$$ \alpha_\phi(z) := \phi(|z|^2) \frac{\alpha(z)\wedge\overline{\alpha(z)}}
 {(-1)^{n(n+1)/2}(2i)^n} .  $$
Let~$A^2_\phi$ and $K_\phi$ be the corresponding weighted Bergman space and
its reproducing kernel, respectively.

\begin{theorem} \label{pPE}
We~have
$$ K_\phi(z,w) = \frac1{(n-1)!c_\MM}
 \Big[ 2t F^{(n-1)}(t) + (n-1) F^{(n-2)}(t) \Big] _{t=z\cb\ow} ,  $$
where
$$ c_\MM = (n-1) \intM \frac{\alpha(z)\wedge\overline{\alpha(z)}} 
 {(-1)^{n(n+1)/2}(2i)^n} $$
and
\[ F(t) = \sumk\frac{t^k}{c_k}, \qquad\text{where }
 c_k := \int_0^\infty t^k \phi(t) \,dt.   \label{tTF}   \]
\end{theorem}

\begin{proof}
It~was shown in \cite[Lemma~2.1]{MY} that for any measurable function $f$
on~$\HH$, 
$$ \intH f(z) \frac{\alpha(z)\wedge\overline{\alpha(z)}}{(-1)^{n(n+1)/2}(2i)^n}
= 2 c_\MM \int_0^\infty \intpM f(t\zeta) \, t^{2n-3} \,d\mu(\zeta) \,dt ,  $$
with $c_\MM$ as above. Thus $\alpha_\phi=\rtm$ for
$$ d\rho(t) = 2 c_\MM t^{2n-3} \phi(t^2) \, dt ,  $$
and by the last theorem, $K_\phi$ is given by \eqref{tTC} with
$$ s_k = \int_0^\infty t^{2k} \,d\rho(t)
 = c_\MM \int_0^\infty t^{k+n-2} \phi(t) \, dt = c_\MM c_{k+n-2}.  $$
Now~by an elementary manipulation,
\begin{align*}
\sum_k \binom{k+n-1}{n-1} \frac{t^k}{c_{k+n-2}}
&= \frac1{(n-1)!} \sum_k \Big(\frac d{dt}\Big)^{n-1}
 \frac{t^{k+n-1}}{c_{k+n-2}}  \\
&= \frac1{(n-1)!} \Big(\frac d{dt}\Big)^{n-1} (tF(t))  \\
&= \frac1{(n-1)!} \Big[ t F^{(n-1)}(t) + (n-1)F^{(n-2)}(t) \Big] , 
\end{align*}
and similarly
\begin{align*}
\sum_k \binom{k+n-2}{n-1} \frac{t^k}{c_{k+n-2}}
&= \frac t{(n-1)!} \sum_k \Big(\frac d{dt}\Big)^{n-1}
 \frac{t^{k+n-2}}{c_{k+n-2}}  \\
&= \frac t{(n-1)!} F^{(n-1)}(t) .
\end{align*}
Thus
$$ \sum_k \frac{N(k) t^k}{c_{k+n-2}} = \frac1{(n-1)!}
  \Big[ 2t F^{(n-1)}(t) + (n-1) F^{(n-2)}(t) \Big]  $$
and the assertion follows. 
\end{proof}

\begin{example}\label{ExA}
Take $\phi(r)=(1-r)^m$ for $r\in[0,1]$ and $\phi(r)=0$ for $r>1$, where $m>-1$.
Then $c_k=\frac{k!\Gamma(m+1)}{\Gamma(m+k+2)}$ and we recover the formula
from~\cite{MY} 
$$ K_\phi(z,w) = \frac{\Gamma(n+m)}{(n-1)!c_\MM\Gamma(m+1)}
 \frac{(n-1)+(n+1+2m)z\cb\ow}{(1-z\cb\ow)^{n+m+1}}   $$
for the ``standard'' Bergman kernels on~$\MM$
(with respect to $\alpha\wedge\overline\alpha$). 
\end{example}

Further applications of Theorem~\ref{pPA} will occur later in the sequel.

We~conclude this section by considering the unit ball $\BB$ in $\CC^n$ with
respect to the ``minimal norm'' given~by
$$ N_\ast (z) = \sqrt{|z|^2 + |z\cb z|}.  $$
This norm was shown to be of interest in the study of several problems related
to proper holomorphic mappings and the Bergman kernel, see \cite{HaP},
\cite{OY}, \cite{OPY}, \cite{MY} and \cite{Mgt}.  

Suppose that $\phi$ is as before and consider the measure $dV_\phi$ on $\CC^n$
with density $\phi(N_\ast^2)$ with respect to the Lebesgue measure. Namely,   
$$ dV_\phi(z) := \phi(|z|^{2} + |z\cb z|) \, dV(z) $$
where $dV(z)$ denotes the Lebesgue measure on $\CC^n$ normalized so that the
volume of the minimal ball is equal to~one. We~denote by $A^2_\phi(\CC^n)$ the
Bergman-Fock type space on $\CC^n$ with respect to~$dV_\phi$, consisting of
all measurable functions $f$ on $\CC^n$ which are holomorphic in the ball
$$ \BB_\phi := \{ z\in \CC^n: \sqrt{|z|^2 +|z\cb z|} < R_\phi \}  $$
and satisfy
\[ \label{equation0} 
 \|f\|_\phi^2 := \int_{\CC^n} |f(z)|^2 dV_\phi(z) < + \infty.  \]
Here $R_\phi$ is the square root of the radius of convergence of the
series~\eqref{tTF}. We~let $L^2_\phi(\CC^n)$ denote the space of all measurable
functions $f$ in $\CC^n$ verifying~\eqref{equation0}. Finally, we~define the
operators $\Delta_j$, $j=0,1$, acting on power series in $z$ by their actions
on the monomials $z^m$ as follows 
$$ \aligned
 (\Delta_0 z^m) (x, y) &:=
  2\sum_{k=0}^{[\frac{m-1}2]} \binom m{2k+1} x^{m-1-2k} y^k,  \\
 (\Delta_1 z^m) (x, y) &:=
  2\sum_{k=0}^{[\frac m2]} \binom m{2k} x^{m-2k} y^k, \qquad x,y\in\CC.
\endaligned  $$ 
If $f(z) =  \sum_k c_k z^k$ is a power series of radius of convergence~$R$,
then the series
$$ (\Delta_j f) (x, y)  := \sum_k c_k (\Delta_j z^k) (x, y^2)  $$
converges as long as $|x|+|y|<R$ and we have
$$ \aligned
 (\Delta_0 f) (x, y) &= \frac{f(x+y) - f(x-y)} y, \qquad y\neq0, \\
 (\Delta_1 f) (x, y) &=  f(x+y) + f(x-y).
\endaligned  $$ 
Using these notations, we~then have the following.

\begin{theorem} \label{thmC}
The space $A^2_\phi(\CC^n)$ coincides with the closure of the holomorphic
polynomials in $L^2_\phi(\CC^n)$ and its reproducing kernel is given~by
\begin{align*}
 K_{\phi,\CC^n}(z,w) &= \frac{(n+1)^2}{(n-1)!c_\MM}
 \Big[ 2 (z\cb\ow)
   \Delta_0 (F^{(n-1)}) (z\cb\ow,z\cb z\cdot\overline{w\cb w})  \\
& \qquad + \Delta_1 (F^{(n-1)}) (z\cb\ow,z\cb z\cdot\overline{w\cb w}) \\
& \qquad + (n-1) \Delta_0 (F^{(n-2)}) (z\cb\ow,z\cb z\cdot\overline{w\cb w}) 
 \Big] ,
\end{align*}
with $F$ as in~\eqref{tTF}.
\end{theorem}

\begin{proof}
We~will use the technique developed in~\cite{MY}.
Let~$\operatorname{Pr}:\CC^{n+1}\to\CC^n$ be the projection onto the first $n$
coordinates
$$ \operatorname{Pr}(z_1,\dots,z_n,z_{n+1}) = (z_1,\dots,z_n)   $$
and $\iota:=\operatorname{Pr}|_\HH$. Then $\iota:\HH\to\CC^n\setminus\{0\}$
is a proper holomorphic mapping of degree~2. The~branching locus of $\iota$
consists of points with $z_{n+1}=0$, and its image under $\iota$ consists of
all $x\in\CC^n\setminus\{0\}$ with $\sum_{j=1}^n x_j^2=0$. The~local inverses
$\Phi$ and $\Psi$ of $\iota$ are given for $z\in\CC^n\setminus\{0\}$~by
\begin{align*}
\Phi(z) &= (z,i\sqrt{z\cb z}), \\
\Psi(z) &= (z,-i\sqrt{z\cb z}). 
\end{align*}
In~view of Lemma~3.1 of~\cite{OPY}, we~see that
\begin{align*}
\Phi^*(\alpha_\phi) &= \frac{1+n}{i\sqrt{z\cb z}} \phi(z)^{1/2}
  (-1)^n dz_1\wedge\dots\wedge dz_n,   \\
\Psi^*(\alpha_\phi) &= \frac{1+n}{-i\sqrt{z\cb z}} \phi(z)^{1/2}
  (-1)^n dz_1\wedge\dots\wedge dz_n.
\end{align*}
If~$f:\CC^n\to\CC$ is a measurable function and $z\in\HH$, we~consider the
operator $U=U_\phi$ by setting
$$ (Uf)(z) := \frac{z_{n+1}}{\sqrt2(n+1)} (f\circ\iota)(z).   $$
Using the same arguments as in the proof of Lemma~4.1 in~\cite{MY}, it~can be
shown that $U$ is an isometry from $L^2_\phi(\CC^n)$ into $L^2_\phi(\HH)$.
More precisely, we~have
$$ \intH |Uf(z)|^2 \alpha_\phi(z) = \int_{\CC^n} |f(z)|^2 \,dV_\phi(z).  $$
In~addition, the~arguments used in the proof of part (2) of the latter lemma
show that the image $\mathcal E_\phi(\HH)$ of $A^2_\phi(\CC^n)$ under $U$ is
a closed proper subspace of $A^2_\phi(\HH)$, and $U$ is unitary from
$A^2_\phi(\CC^n)$ onto~$\mathcal E_\phi(\HH)$.
From the technique used in the proof of Lemma~2.4 in~\cite{MY}, we~the get
the following lemma.

\begin{lemma}
If~$\Phi$ and $\Psi$ are as before, then
$$ z_{n+1} K_{\phi,\CC^n}(\iota z,w) = (n+1)^2 \left[
 \frac{K_\phi(z,\Phi(w))}{\overline{\Phi_{n+1}(w)}}
+\frac{K_\phi(z,\Psi(w))}{\overline{\Psi_{n+1}(w)}} \right],   $$
for all $z\in\HH$, $w\in\CC^n$.
\end{lemma}

The~rest of the proof of the theorem now follows from the last lemma and the
identities used in the proof of Theorem~A in~\cite{MY}.
\end{proof}

\begin{example} \label{ExB}
Let $\phi(r)=e^{-cr}$, $c>0$. Then $s_k=k!/c^{k+1}$, $F(t)=ce^{ct}$
and we obtain
\begin{align*}
K_{\phi,\CC^n}(z,w) &= \frac{2(n+1)^2 c^{n-2}} {(n-1)!c_\MM} \times \quad \\
& \qquad
 e^{cz\cb\ow} \left[ 2c^2 (n-1+z\cb\ow) S(c^2 z\cb z\overline{w\cb w})
 + c C(c^2 z\cb z\overline{w\cb w}) \right], 
\end{align*}
where we wrote for brevity $S(t)=\frac{\sinh\sqrt t}{\sqrt t}$,
$C(t)=\cosh\sqrt t$.
\end{example}

\section{TYZ expansions}
\newtheorem*{proofB}{Proof of Theorem~\ref{pPB}}
\begin{proofB}
For~the measure~\eqref{tDR}, we~have by the change of variable $x=st^{2m}$,
$$ q_{2k} = 2cm\int_0^\infty t^{2k} e^{-st^{2m}} \, t^{2mn-1}\,dt 
= \frac cs \int_0^\infty \Big(\frac xs\Big)^{\frac{2k+2mn}{2m}-1} e^{-x}\,dx
= c \frac{\Gamma(\frac{k+mn}m)}{s^{\frac{k+mn}m}} .  $$
Applying Theorem~\ref{pPA}, we~get
\begin{align}
K_s(x,y) &= \sum_k \frac {s^{\frac{mn+k}m}(x\cb\oy)^kN(k)}
  {c\Gamma(\frac{k+mn}m)}   \nonumber   \\
&= \frac{s^n}{(n-1)!c} \Big[ \Big(\frac d{dt}\Big)^{n-1} t^{n-1}
 + t \Big(\frac d{dt}\Big)^{n-1} t^{n-2} \Big] E_{\frac1m,n}(t)
 \Big| _{t=s^{1/m} x\cb\oy} ,   \label{tEZ}
\end{align}
by~a~similar computation as in the proof of Theorem~\ref{pPE}.
Here $E_{\frac1m,n}$ is the Mittag-Leffler function
$$ E_{\alpha,\beta}(z) := \sumk \frac{z^k}{\Gamma(\alpha k+\beta)},
 \qquad \alpha,\beta>0.   $$
Recall now that as $z\to\infty$, $E_{\alpha,\beta}$ has the asymptotics
$$ E_{\alpha,\beta}(z) = \begin{cases}
 \frac1\alpha \sum_{j=-N}^N z_j^{1-\beta} e^{z_j} + O\Big(\frac1z\Big) , 
   \qquad & |\arg z|\le\frac{\pi\alpha}2,  \\
 O\Big(\frac1z\Big), & |\arg(-z)|<\pi-\frac{\pi\alpha}2,   \end{cases}  $$
where $N$ is the integer satisfying $N<\frac\alpha2\le N+1$ and
$z_j=|z|^{1/\alpha}e^{(\arg z+2\pi i j)/\alpha}$ with $-\pi<\arg z\le\pi$.
See~e.g.~\cite[\S18.1, formulas (21)--(22)]{BErd}
(additional~handy references are given in Section~7 of~\cite{GFock}).
Furthermore, this asymptotic expansion can be differentiated termwise any
number of times (see~again Section~7 in~\cite{GFock} for details on~this).
In~particular, for $t>0$ the term $j=0$ dominates all the others,
and~we therefore obtain
\[ E_{\alpha,\beta}(t) = \frac1\alpha t^{(1-\beta)/\alpha} e^{t^{1/\alpha}}
 + O(e^{(1-\delta)t^{1/\alpha}}) \qquad\text{as } t\to+\infty  \label{tEA} \]
with some $\delta>0$.
(One~can take any $0<\delta<\frac1{\sqrt2}(1-\cos\frac{2\pi}\alpha)$
for $\alpha>4$, and any $0<\delta<\frac1{\sqrt2}$ for $0<\alpha<4$.)
Moreover, \eqref{tEA} remains in force when a derivative of any order is
applied to the left-hand side and to the first term on the right-hand~side.

Now~by a simple induction argument,
\[ \Big(\frac d{dt}\Big)^k t^{\gamma m}e^{t^m}
 = t^{\gamma m-k} e^{t^m} p_k(t^m),  \label{tEB}   \]
where $p_k$ are polynomials of degree $k$ defined recursively~by
\[ p_0=1, \quad p_k(x) = (\gamma m-k+1+mx) p_{k-1}(x) + mx p'_{k-1}(x).
 \label{tER}   \]
A~short computation reveals that
\[ p_k(x) = m^k x^k + [k m^k\gamma + \tfrac{k(k-1)}2 (m-1) m^{k-1}] x^{k-1}
 + \dots .   \label{tEC}   \]
Taking $\alpha=\frac1m$, $\beta=n$, $\gamma=1-n$, an~application of the Leibniz
rule shows that
\begin{align}
& \Big[ \Big(\frac d{dt}\Big)^{n-1} t^{n-1}
 + t \Big(\frac d{dt}\Big)^{n-1} t^{n-2} \Big] E_{\frac1m,n}(t)  \nonumber  \\
& \hskip3em
 = \sum_{j=0}^{n-1} \binom{n-1}j E_{\frac1m,n}^{(j)}(t)
 \Big[ \frac{(n-1)!}{j!}t^j + \frac{(n-2)!}{(j-1)!}t^j \Big]   \nonumber  \\
& \hskip3em
 = \sum_{j=0}^{n-1} \binom{n-1}j 
 \Big[ \frac{(n-1)!}{j!} + \frac{(n-2)!}{(j-1)!} \Big]
 m t^{(1-n)m} e^{t^m} p_j(t^m) + O(e^{(1-\delta)t^m}),   \label{tES}
\end{align}
so
\begin{align}
& e^{-t^m} \Big[ \Big(\frac d{dt}\Big)^{n-1} t^{n-1}
 + t \Big(\frac d{dt}\Big)^{n-1} t^{n-2} \Big] E_{\frac1m,n}(t)  \nonumber  \\
& \hskip4em
 = \sum_{j=0}^{n-1} \binom{n-1}j 
 \Big[ \frac{(n-1)!}{j!} + \frac{(n-2)!}{(j-1)!} \Big]
 m t^{(1-n)m} p_j(t^m) + O(e^{-\delta t^m})   \nonumber   \\
& \hskip4em
 = \sum_{k=0}^{n-1} \frac{2m^n b_k}{t^{km}} + O(e^{-\delta t^m}) \label{tED}
\end{align}
as $t\to+\infty$, where $b_k$ are some constants depending only on~$m,n$,
with (after a small computation)
$$ b_0=1, \qquad b_1 = -\frac{(n-1)(mn-n+1)}{2m}   $$
by~\eqref{tEC}. Setting $t=s^{1/m}|z|^2$ and substituting \eqref{tED}
into~\eqref{tEZ}, we~are~done.
\qed  \end{proofB}

As~remarked in the Introduction, in the context of the TYZ expansions
Theorem~\ref{pPB} corresponds to the situation of the trivial bundle
$\HH\times\CC$ over~$\HH$, with Hermitian metric on the fiber given by
$h(z)=e^{-s|z|^{2m}}$. The~associated K\"ahler form $\omega=\frac i2\ddlh
=\frac{is}2\partial\dbar|z|^{2m}$ can be computed similarly as in \cite{Rawn}
for $\frac i2\partial\dbar|z|$. Even without that, however, one~can see what
is the corresponding volume element~$\omega^n$: namely, since differentiation
lowers the degree of homogeneity by~1, the~density of $\omega^n$ with respect
to the Euclidean surface measure on $\HH$ must have homogeneity $n\cdot(2m-2)
=2mn-2n$; and~as the surface measure equals, up~to a constant factor,
to $t^{2n-1}dt\otimes d\mu$, we~see that $\omega^n=\drtm$ for $d\rho(t)=
ct^{2mn-1}\,dt$, with some constant~$c$. Thus
$$ e^{-s|z|^{2m}} K_s(z,z) \equiv \epsilon_s (z)   $$
is~indeed precisely the Kempf distortion function for the above line bundle,
and \eqref{tDQ} is its asymptotic, or~TYZ, expansion.

The~bundles studied by Gramchev and Loi in \cite{GrL} with $\omega=\frac i2
\partial\dbar|z|$ correspond to $m=\frac12$. Note that in that case,
in~agreement with~\cite{GrL}, the~lowest-order term in \eqref{tDQ} actually
vanishes, i.e.~$b_{n-1}=0$. In~fact, the~constant term in $p_k$ equals,
by~\eqref{tER},
$$ p_k(0) = (\gamma m-k+1)(\gamma m-k+2)\dots(\gamma m)
 = (-1)^k (-\gamma m)_k   $$
(where $(\nu)_k:=\nu(\nu+1)\dots(\nu+k-1)$ is the usual Pochhammer symbol);
thus for the lowest order term in the sum in \eqref{tED} we get
from~\eqref{tES} 
\begin{align*}
2m^n b_{n-1} &= m \sum_{j=0}^{n-1} \binom{n-1}j 
 \Big[ \frac{(n-1)!}{j!} + \frac{(n-2)!}{(j-1)!} \Big] (-1)^j (-\gamma m)_j  \\
&= (n-1)m(1-2m) \prod_{j=1}^{n-2} (j-(n-1)m) 
\end{align*}
after some computation (using the Chu-Vandermonde identity), which vanishes for
$m=\frac12$. This explains why the summations stops at $k=n-2$ in~\eqref{tDY}. 

We~remark that in a completely analogous manner, one~could also derive the
asymptotics as $s\to+\infty$ of the reproducing kernels for the same weights
$e^{-s|z|^{2m}}$ but with respect to the density $\frac{\alpha(z)\wedge
\overline{\alpha(z)}} {(-1)^{n(n+1)/2}(2i)^n}$ instead of $(\frac i2\partial
\dbar|z|^{2m})^n$. By~Theorem~\ref{pPE} the corresponding kernels are given~by
\[ K_s(z,w) = \frac{m s^{\frac{n-1}m}} {(n-1)!c_\MM} 
 \Big[ 2t \Big(\frac d{dt}\Big)^{n-1} + (n-1) \Big(\frac d{dt}\Big)^{n-2} \Big]
 E_{\frac1m,\frac1m}(t) \Big|_{t=s^{\frac1m}z\cb\ow}  \label{tEX}   \]
and the needed asymptotics of derivatives of $E_{\frac1m,\frac1m}$ are worked
out e.g.~in Section~7 of~\cite{GFock} (or~can be obtained from the formulas
above in this section). We~leave the details to the interested reader.
For~$m=1$, \eqref{tEX}~recovers the formula from Gonessa and Youssfi~\cite{GY}.

Finally, one~can establish analogous ``TYZ'' asymptotics also for Bergman-type
kernels on the unit ball $\MM$ of~$\HH$; we~limit ourselves to the following
variant of Theorem~3.2 from~\cite{MY}.

\begin{theorem} \label{pPF}
For~the weights corresponding~to
$$ d\rho(t) = (1-t^2)^s t^{2n-1} \,dt, \qquad s>-1   $$
$($i.e.~having the density $(1-|z|^2)^s$ with respect to the Euclidean surface
measure$)$ on~$\MM$, the~corresponding reproducing kernels $K_s$ of $A^2_\rtm$
are given~by
\begin{align*}
& K_s(z,w) = \frac{\Gamma(n+s+1)}{(n-1)!\Gamma(s+1)}  \times \quad \\
& \hskip4em
 \Big[ \frac1{(1-t)^{n+s+1}} + \frac{n+s+1}{t^{n-1}}
 \sum_{j=0}^{n-1} \binom{n-1}j (-1)^j \frac{(1-t)^{j-n-s-1}-1}{n+s+1-j}
 \Big] _{t=z\cb\ow} .
\end{align*}
\end{theorem}

\begin{proof} Since
$$ q_{2l} = \int_0^1 t^{2l}(1-t^2)^s t^{2n-1}\,dt
 = \frac{\Gamma(s+1)\Gamma(l+n)}{\Gamma(l+n+s+1)} ,   $$
we~get from Theorem~\ref{pPA}
$$ K_s(z,w) = \suml \frac{N(l) t^l \Gamma(l+n+s+1)} {(l+n-1)!\Gamma(s+1)}  $$
where we have set for brevity $t=z\cb\ow$. Hence
\begin{align*}
\frac{(n-1)!\Gamma(s+1)} {\Gamma(n+s+1)} K_s(z,w)
&= \suml t^l\Big[ \frac{(n+s+1)_l}{l!} + \frac{(n+s+1)_l}{(l-1)!(l+n-1)}\Big]\\
&= \frac1{(1-t)^{n+s+1}} + \sumk \frac{(n+s+1)_{k+1} t^{k+1}} {k!(n+k)}  \\
&= \frac1{(1-t)^{n+s+1}} + (n+s+1)t \sumk \frac{(n+s+2)_k t^k}{k!(n+k)}.
\end{align*}
The~last sum can be written~as
\begin{align*}
\sumk \frac{(n+s+2)_k t^k}{k!(n+k)}
&= \frac1{t^n} \int_0^t \sumk \frac{(n+s+2)_k}{k!} x^{k+n-1} \, dx  \\
&= \frac1{t^n} \int_0^t \frac{x^{n-1}}{(1-x)^{n+s+2}} \, dx  \\
&= \frac1{t^n} \int_0^t \sum_{j=0}^{n-1} \binom{n-1}j
 \frac{(x-1)^j} {(1-x)^{n+s+2}} \, dx  \\
&= \frac1{t^n} \sum_{j=0}^{n-1} \binom{n-1}j (-1)^j
 \frac{(1-t)^{j-n-s-1}-1} {s+n+1-j} , 
\end{align*}
proving the theorem.
\end{proof}

\begin{corollary}  \label{pPG}
In~the setup from the last theorem,
\[ (1-|z|^2)^{s+n+1} K_s(z,z) \approx s^n \sumk \frac{a_k(|z|^2)}{s^k}
 \qquad\text{as } s\to+\infty,  \label{tET}  \postdisplaypenalty1000000 \]
where $a_k(|z|^2)$ are some functions, with $a_0=2$.
\end{corollary}

\begin{proof}
With $t=|z|^2$, Theorem~\ref{pPF} shows that the left-hand side equals
$$ \frac{(s+1)_n}{(n-1)!} \Big[ 1+ \frac{n+s+1}{t^{n-1}} \sum_{j=0}^{n-1}
 \binom{n-1}j (-1)^j \frac{(1-t)^j-(1-t)^{s+n+1}} {n+s+1-j} \Big].   $$
The~term $(1-t)^{n+s+1}$ is exponentially small compared to the rest,
and thus can be neglected. Noticing that $(s+1)_n$ is a polynomial in $s$
of degree $n$ while
$$ \frac{n+s+1}{n+s+1-j} = 1 + \frac j{s(1+\frac{n+1-j}s)}
 = 1 + \sumk \frac{(j-n-1)^k j}{s^{k+1}} ,  $$
the~expansion \eqref{tET} follows, as~does the formula for $a_0$ upon
a small computation.
\end{proof}

\section{Hankel operators}
\newtheorem*{proofC}{Proof of Theorem~\ref{pPC}}
\begin{proofC}
(i)$\implies$(ii) Let $g=\sum_k g_k$ be the homogeneous expansion \eqref{tDC}
of~$g$. For~$\epsilon\in\TT$, consider the rotation operator
$$ U_\epsilon f(z) := f(\epsilon z), \qquad z\in\MM.   $$
Clearly $U_\epsilon$ is unitary on $A^2(\MM)$ as well as on~$L^2(\MM)$, and
$$ U_\epsilon \Ho g U_\epsilon^* = \Ho{U_\epsilon g} .  $$
Thus also $\Ho{U_\epsilon g}\in\cS^p$ for all $\epsilon\in\TT$.
Furthermore, the~action $\epsilon\mapsto U_\epsilon$ is continuous in the
strong operator topology, i.e.~$\epsilon\mapsto U_\epsilon f$ is norm
continuous for each $f\in L^2(\MM)$. Now~it was shown in Lemma on p.~997
in~\cite{AFPD} that this implies that the map $\epsilon\mapsto
\Ho{U_\epsilon g}$ is~even continuous from $\TT$ into~$\cS^p$.
Consequently, the~Bochner integral
$$ \int_0^{2\pi} e^{mi\theta} U_{e^{i\theta}} \Ho g U^*_{e^{i\theta}}
 \, \frac{d\theta}{2\pi} = \Ho{g_m}   $$
also belongs to~$\cS^p$, for~each~$m$. As~$g$ is nonconstant, there exists
$m\ge1$ for which $g_m\neq0$, and (ii) follows.

(ii)$\implies$(iii) Assume that $\Ho g\in\cS^p$ for some $0\neq g\in\PP_m$.
For~any transform $\kappa\in\OnR$, the~corresponding composition
$$ U_\kappa f(z) := f(\kappa z), \qquad z\in\MM,   $$
again acts unitarily on $A^2(\MM)$ as well as on~$L^2(\MM)$, and
$$ U_\kappa \Ho g U^*_\kappa = \Ho{U_\kappa g} ,   $$
so~that also $\Ho{U_\kappa g}\in\cS^p$ for all $\kappa\in\OnR$.
Likewise, the~composition $V_\kappa:f\mapsto f\circ\kappa$ sends the space
$\cH^m$ of spherical harmonics into itself, and~the isomorphism \eqref{tDW}
satisfies $\widehat{V_\kappa f}=U_\kappa\hat f$. Now~it is known \cite{Mull}
that the representation $\kappa\mapsto V^{-1}_\kappa$ of $\OnR$ on $\cH^m$
is irreducible, that~is, for~any nonzero $f\in\cH^m$ the span of the
translates $V_\kappa f$, $\kappa\in\OnR$, is~dense in~$\cH^m$.
Consequently, for~any nonzero $g\in\PP_m$, the~span of the translates
$U_\kappa g$, $\kappa\in\OnR$, of~$g$ by $\kappa$ is dense in~$\PP_m$.
As~$\PP_m$ has finite dimension~$N(m)$, this actually means that $\PP_m$
consists just of all linear combinations of $U_{\kappa_j}g$ for some
tuple $\kappa_1,\dots,\kappa_{N(m)}$ in~$\OnR$. Since $\Ho{U_{\kappa_j}g}
\in\cS^p$ for each~$j$, it~follows by linearity that $\Ho h\in\cS^p$ for all
$h\in\PP_m$, showing that (iii) holds (for~the same $m$ as in~(ii)).

(iii)$\implies$(iv) By~hypothesis, we~have in particular $\Ho{z^\nu}\in\cS^p$
for all multiindices $\nu$ with $|\nu|=m$; thus the operator
\[ H := \sum_{|\alpha|=m} \binom m\alpha \Ho{z^\alpha}^* \Ho{z^\alpha}
  \label{tHC}  \]
belongs to $\cS^p\cdot\cS^p=\cS^{p/2}$. 

Recall that the Toeplitz operator~$T_\phi$, $\phi\in L^\infty(\MM)$,
is~the operator on $A^2(\MM)$ defined~by
$$ T_\phi f = P(\phi f)   $$
($P$~being, as~before, the~orthogonal projection in $L^2$ onto~$A^2$).
One~thus~has
\[ \|H_\phi f\|^2 = \|\phi f\|^2 - \|T_\phi f\|^2 , \label{tHA}   \]
and, by~the reproducing property of the Bergman kernels,
$$ T_\phi f(x) = \intM f(y) \phi(y) K(x,y) \, dy,   $$
where we started to write for brevity (for~the duration of this proof)
just~$dy$ instead of $\drtm(y)$. With the notation from Theorem~\ref{pPA},
we~thus have, for~any $\xi\in\MM$,
\begin{align*}
T_{\overline{z^\alpha}} (z\cb\xi)^l (x)
&= \intM (y\cb\xi)^l \oy^\alpha \sum_k \frac{(x\cb\oy)^k}{d_k} \, dy  \\
&= \sum_k \intM (y\cb\xi)^l \partial_x^\alpha \frac{k!}{(k+|\alpha|)!}
 \frac{(x\cb\oy)^{k+|\alpha|}}{d_k} \, dy   \\
&= \sum_k \frac{k!}{(k+|\alpha|)!d_k} \partial_x^\alpha
 \intM (y\cb\xi)^l (x\cb\oy)^{k+|\alpha|} \,dy   \\
&= \sum_k \frac{k!}{(k+|\alpha|)!d_k} \partial_x^\alpha
 \int_0^1 t^{l+k+|\alpha|} \, d\rho(t)
 \intpM (y\cb\xi)^l (x\cb\oy)^{k+|\alpha|} \,d\mu(y)   \\
&= \sum_k \frac{k!}{(k+|\alpha|)!d_k} \partial_x^\alpha
 \delta_{l,k+|\alpha|} d_l (x\cb\xi)^l
\end{align*}
by~\eqref{tTA} and~\eqref{tDL}. This vanishes for $l<|\alpha|$,
while for $l\ge|\alpha|$ it equals
$$ \frac{(l-|\alpha|)!}{l!} \frac{d_l}{d_{l-|\alpha|}} \partial_x^\alpha
 (x\cb\xi)^l = \frac{d_l}{d_{l-|\alpha|}} \xi^\alpha (x\cb\xi)^{l-|\alpha|}. $$
Declaring $d_l$ to be $\infty$ for negative~$l$, we~thus obtain for all~$l$
$$ T_{\overline{z^\alpha}} (z\cb\xi)^l = 
\frac{d_l}{d_{l-|\alpha|}} \xi^\alpha (z\cb\xi)^{l-|\alpha|}.  $$
For~any $\xi,\eta\in\MM$, we~then get by~\eqref{tHA}
\begin{align*}
&\spr{H(\cdot\cb\xi)^l,(\cdot\cb\eta)^k}
= \intM \sum_{|\alpha|=m} \binom m\alpha
  |z^\alpha|^2 (z\cb\xi)^l (\oz\cb\oeta)^k \, dz  \\
& \hskip8em
 - \frac{d_l^2}{d_{l-m}^2} \intM \sum_{|\alpha|=m} \binom m\alpha
  \xi^\alpha \overline{\eta^\alpha} (z\cb\xi)^{l-m}(\oz\cb\oeta)^{k-m}\,dz  \\
& \hskip4em
 = \intM |z|^{2m} (z\cb\xi)^l (\oz\cb\oeta)^k \, dz
 - \frac{d_l^2}{d_{l-m}^2} (\xi\cb\oeta)^m
   \intM (z\cb\xi)^{l-m}(\oz\cb\oeta)^{k-m} \, dz  \\
& \hskip4em
 = \int_0^1 t^{2m+k+l}\,d\rho(t)
   \intpM (z\cb\xi)^l (\oz\cb\oeta)^k \, d\mu(z)  \\
& \hskip8em
 - \frac{d_l^2}{d_{l-m}^2} (\xi\cb\oeta)^m
   \int_0^1 t^{l+k-2m}\,d\rho(t)
   \intpM (z\cb\xi)^{l-m}(\oz\cb\oeta)^{k-m} \, d\mu(z)  \\
& \hskip4em
 = \delta_{kl} \frac{q_{2l+2m}}{N(l)} (\xi\cb\oeta)^l
 - \frac{d_l^2}{d_{l-m}^2} (\xi\cb\oeta)^m
   \delta_{kl} q_{2l-2m} \frac{(\xi\cb\oeta)^{l-m}}{N(l-m)}
\end{align*}
by~\eqref{tTA} and~\eqref{tDL} again. Comparing this with
\begin{align*}
\spr{(\cdot\cb\xi)^l,(\cdot\cb\eta)^k}
&= \int_0^1 t^{l+k}\,d\rho(t) \intpM (z\cb\xi)^l (\oz\cb\oeta)^k \, d\mu(z) \\
&= \delta_{kl} \frac{q_{2l}}{N(l)} (\xi\cb\oeta)^l
= \delta_{kl} d_l (\xi\cb\oeta)^l,  
\end{align*}
we~thus see that, for~all~$l$,
$$ \spr{H(\cdot\cb\xi)^l,(\cdot\cb\eta)^k} =
 \Big( \frac{q_{2l+2m}}{q_{2l}} - \frac{d_l}{d_{l-m}} \Big)
 \spr{(\cdot\cb\xi)^l,(\cdot\cb\eta)^k} .   $$
Since $(\cdot\cb\xi)^l$, $\xi\in\MM$, span~$\PP_l$, it~follows that
$$ \spr{Hf_l,g_k} =  \Big( \frac{q_{2l+2m}}{q_{2l}} - \frac{d_l}{d_{l-m}} \Big)
 \spr{f_l,g_k}   $$
for any $f_l\in\PP_l$ and $g_k\in\PP_k$. In~other~words, the~operator $H$ is
diagonalized by the orthogonal decomposition $A^2(\MM)=\bigoplus_k\PP_k$, and
$$ H = \bigoplus_l  \Big( \frac{q_{2l+2m}}{q_{2l}} - \frac{d_l}{d_{l-m}} \Big)
 \; I|_{\PP_l}.   $$
Recalling that $\dim\PP_k=N(k)$, this means that
\[ H\in\cS^{p/2} \iff \sum_l 
 \Big( \frac{q_{2l+2m}}{q_{2l}} - \frac{d_l}{d_{l-m}} \Big)^{p/2} N(l)
 < \infty .   \label{tHB}   \]
Now~for our measure $d\rho(t)=\chi_{[0,1]}(t)t^{2n-1}\,dt$,
one~has $q_{2k}=\frac1{2k+2n}$,~so
\begin{align*}
d_k &= \frac{q_{2k}}{N(k)} = \frac{(n-1)!k!}{2(k+n)(2k+n-1)(k+n-2)!}  \\
&= \frac{(n-1)!}{4k^n} \Big[1- \frac{n+\frac{n-1}2+\frac{(n-1)(n-2)}2}k
 + O\Big(\frac1{k^2}\Big) \Big] 
\end{align*}
and
$$ \frac{d_k}{d_{k-m}} = 1 - \frac{mn}k + O\Big(\frac1{k^2}\Big)  $$
as $k\to+\infty$, while
$$ \frac{q_{2k+2m}}{q_{2k}} = \frac{k+n}{k+n+m}
 = 1-\frac mk + O\Big(\frac1{k^2}\Big) .   $$
Thus 
$$ \Big( \frac{q_{2k+2m}}{q_{2k}} - \frac{d_k}{d_{k-m}} \Big)
 \sim \frac{(n-1)m}k   $$
and the sum \eqref{tHB} is finite if and only~if
(recall that $n\ge2$ and $m>0$)
$$ +\infty > \sum_k k^{-p/2} N(k) \sim \sum_k k^{n-1-p/2}   $$
i.e.~if~and only~if $p>2n$. Thus we have proved~(iv).

(iv)$\implies$(v) We~have seen in the last computation that for $p>2n$,
the~operator \eqref{tHC} belongs to $\cS^{p/2}$ for any $m\ge1$.
Since all the summands $\Ho{z^\alpha}^*\Ho{z^\alpha}$ in~\eqref{tHC}
are nonnegative operators, it~follows that even $\Ho{z^\alpha}^*\Ho{z^\alpha}
\in\cS^{p/2}$ for all~$\alpha$, i.e.~$\Ho{z^\alpha}\in\cS^p$ for all~$\alpha$.
By~linearity, $\Ho g\in\cS^p$ for any polynomial~$g$, proving~(v).

The~last implication (v)$\implies$(i) is trivial, and thus the proof of
Theorem~\ref{pPC} is complete.
\qed  \end{proofC}

Note that the only place where $d\rho$ entered was in the computation in
the part (iii)$\implies$(iv). Thus~Theorem~\ref{pPC} remains in force also
for any measure $d\rho$ for which
$$ q_{2k} = a k^r [1+\tfrac bk + O(k^{-2}) ]  \quad\text{as }k\to+\infty   $$
for some $a\neq0$ and $b,r\in\RR$; because one then has
$\frac{q_{2k+2m}}{q_{2k}}=1+\frac{mr}k+O(k^{-2})$, 
$\frac{d_k}{d_{k-m}}=1+\frac{(r-n+1)m}k+O(k^{-2})$
and again $(\frac{q_{2k+2m}}{q_{2k}}-\frac{d_k}{d_{k-m}})\sim\frac{(n-1)m}k$.
In~particular, Theorem~\ref{pPC} thus also holds for the ``standard'' weighted
Bergman spaces on $\MM$ corresponding to
$d\rho(t)=\chi_{[0,1]}(t) (1-t^2)^\alpha t^{2n-1}\,dt$, $\alpha>-1$.

\section{Balanced metrics}
Returning to the formalism reviewed in connection with the TYZ expansions,
recall that a K\"ahler form $\omega=\frac i2\ddlh$ on a domain in $\CC^n$
(or~the Hermitian metric associated to~$\omega$) is called \emph{balanced}
if~the corresponding weighted Bergman kernel satisfies
\[ h(z) K_{h\det[\ddlh]}(z,z) \equiv\text{const.} \quad(\neq0); \label{tBA} \]
that~is, if~and only~if the corresponding Kempf distortion function
$\epsilon_1$ is a nonzero constant. More~generally, for $\alpha>n$ one calls
$\omega$ \emph{$\alpha$-balanced} if it is balanced and $h$ is commensurable
to $\operatorname{dist}(\cdot,\partial\Omega)^\alpha$ at the boundary.
(For~$\alpha\le n$, the~corresponding Bergman space degenerates just to
the zero function, thus $K_{h\det[\ddlh]}\equiv0$ and the left-hand side
in~\eqref{tBA} is constant zero.) This~definition turns out to indeed depend
only on $\omega=\frac i2\ddlh$ and not on the particular choice of $h$ for
a given~$\omega$, and also can be extended from domains to manifolds with
line bundles as before; see~Donaldson~\cite{Don}, Arezzo and Loi \cite{ArL}
and \cite{EKok} for further details.

The~simplest example of $\alpha$-balanced metric is $h(z)=(1-|z|^2)^\alpha$,
$\alpha>1$, on~the unit disc $\DD$ in~$\CC$; one~then gets $\det[\ddlh]=
\frac\alpha{(1-|z|^2)^2}$ and $K_{h\det[\ddlh]}(z,z)=\frac{\alpha-1}
{\pi\alpha}(1-|z|^2)^{-\alpha}$, so~that \eqref{tBA} holds with the constant
$\frac{\alpha-1}{\pi\alpha}$. Similarly for $h(z)=(1-|z|^2)^\alpha$,
$\alpha>n$, on~the unit ball $\BB^n$ of~$\CC^n$, where the constant turns
out to be $\frac{\Gamma(\alpha)}{\Gamma(\alpha-n)\alpha^n\pi^n}$.
The~only known examples of bounded domains with balanced metrics
are invariant metrics on bounded homogeneous domains
(in~particular, on~bounded symmetric domains).
In~the unbounded setting, every dilation of the Euclidean metric on $\CC^n$
is balanced (with $h(z)=e^{-\alpha|z|^2}$, $\alpha>0$, so~that $A^2(\CC^n,
h\det[\ddlh])$ is just the familiar Fock space). Balanced metrics are known
to exist in abundance on compact manifolds~\cite{Don}; the~existence of
balanced metrics e.g.~on bounded strictly pseudoconvex domains with
smooth boundary is still an open problem, and their uniqueness for
a given $\alpha$ is an open problem even on the unit~disc.

We~conclude this paper by the following simple observation concerning
balanced metrics on our Kepler manifold~$\HH$.
Note that for $n\ge3$, $\HH$~is known to be simply connected.

\begin{theorem} \label{pPH}
Let $n=\dim_\CC\HH\ge3$ and $\alpha>n$. Then either there does not exist any
$\alpha$-balanced metric on~$\HH$, or~it~is not~unique.
\end{theorem}

\begin{proof}
Suppose there exists a unique $\alpha$-balanced K\"ahler form
$\omega=\frac i2\ddlh$ on~$\HH$. It~was shown in \cite{EKok} that the
image of an $\alpha$-balanced metric on a simply connected domain under
a biholomorphic map is again $\alpha$-balanced (with the same~$\alpha$).
(Simple connectivity is needed for some powers of Jacobians to be
single-valued when $\frac\alpha{n+1}$ is not an integer.)
Consequently, the~K\"ahler form~$\omega$, or, equivalently, the~associated
Hermitian metric~$g_{j\overline k}$, at~the point $e=(1,i,0,\dots,0)\in\HH$
has to be invariant under any biholomorphic self-map of $\HH$ fixing~$e$,
in~particular, under all elements of the isotropy subgroup
$L:=\{\kappa\in\OnC:\kappa e=e\}$ of $e$ in~$\OnC$.
The~unit sphere with respect to $g_{j\overline k}$ in the tangent space
$T_e\HH$ of $\HH$ at $e$ thus has to be invariant under~$L$.
Now $g_{j\overline k}$ being a Riemannian metric, the~unit sphere
is diffeomorphic to~$\partial\BB^n\cong\bS^{2n-1}$, in~particular,
it~is~a~compact manifold. However, $L$~is easily seen to have orbits
that extend to infinity. For~instance, the~matrix
$$ A_z := \begin{pmatrix}
  {\begin{matrix} 1+\frac{z^2}2 & \frac{iz^2}2 & iz \\
   \frac{iz^2}2 & 1-\frac{z^2}2 & -z \\
   -iz & z & 1 \end{matrix}} & \vrule & 0  \\
   \noalign{\hrule}
    0  & \vrule & I  \end{pmatrix}    $$
is~easily checked to belong to $L$ for any $z\in\CC$, while the orbit of
$\overline e=(1,-i,0,\dots,0)\in T_e\HH$ under~$A_z$,
$$ A_z \overline e = (z^2+1,i(z^2-1),-2iz)    $$
evidently does not stay in any compact subset as $z\to\infty$.
We~have reached a contradiction.
\end{proof}

The~same argument applies, in~principle, to~any manifold whose isotropy
subgroup $L$ of biholomorphic self-maps that stabilize some basepoint $e$
contains a complex one-parameter subgroup (such~as the $A_z$, $z\in\CC$, 
above): by~Liouville's theorem, the~orbit of any tangent vector under $L$
cannot stay bounded without being constant, and it follows that each
element of $L$ acts trivially (i.e.~reduces to the identity) on~the
tangent space at~$e$. Consequently, if~there exists $\kappa\in L$ with
$d\kappa_e\neq I$, then for each~$\alpha$, the~manifold either does not
admit any $\alpha$-balanced metric, or~such metric is not unique.


\begin{thebibliography}{99}

\bibitem{AFPB} J. Arazy, S. Fisher, S. Janson, J. Peetre:
{\it Membership of Hankel operators on the ball in unitary ideals,\/}
J. Lond. Math. Soc. II. Ser. {\bf 43} (1991), 485--508.

\bibitem{AFPD} J. Arazy, S.D. Fisher, J. Peetre: {\it Hankel operators on
weighted Bergman spaces,\/} Amer. J. Math. {\bf 110} (1988), 989 -- 1054.

\bibitem{ArL} C. Arezzo, A. Loi: {\it Moment maps, scalar curvature and
quantization of K\"ahler manifolds,\/} Comm. Math. Phys. {\bf 246} (2004),
543--559.

\bibitem{BErd} H. Bateman, A. Erd\'elyi, {\it Higher transcendental functions,
vol.~III,\/} McGraw-Hill, New~York-Toronto-London, 1953.

\bibitem{Ber} F.A. Berezin: {\it Quantization,\/}
Math. USSR Izvestiya {\bf 8} (1974), 1109--1163.

\bibitem{GFock} H. Bommier-Hato, M. Engli\v s, E.-H. Youssfi:
{\it Bergman-type projections on generalized Fock spaces,\/}
J.~Math. Anal. Appl. {\bf389} (2012), 1086--1104. 

\bibitem{Catl} D. Catlin: {\it The Bergman kernel and a theorem of Tian,\/}
Analysis and geometry in several complex variables (Katata, 1997),
Trends in Math., pp.~1--23, Birkh\"auser, Boston, 1999.

\bibitem{Don} S.K. Donaldson: {\it Scalar curvature and projective
embeddings~I,\/} J.~Diff. Geom. {\bf 59} (2001), 479--522.

\bibitem{E32} M. Engli\v s: {\it Weighted Bergman kernels and quantization,\/}
Comm. Math. Phys. {\bf 227} (2002), 211--241.

\bibitem{EKok} M. Engli\v s: {\it Weighted Bergman kernels and balanced
metrics,\/} RIMS Kokyuroku {\bf 1487} (2006), 40--54.

\bibitem{GrL} T. Gramchev, A. Loi: {\it TYZ expansion for the Kepler
manifold,\/} Comm. Math. Phys. {\bf 289} (2009), 825--840.

\bibitem{GY} J. Gonessa, E.-H. Youssfi: {\it The Bergman projection in
spaces of entire functions,\/} Ann. Polon. Math. {\bf 104} (2012), 161--174.

\bibitem{HaP} K.T. Hahn, P. Pflug: {\it On a minimal complex norm that
extends the real Euclidean norm,\/} Monatsh. Math. {\bf 108} (1998), 107--112.

\bibitem{Ii} K. Ii: {\it On Bargmann-type transform and a Hilbert space of
holomorphic functions,\/} T\^ohoku Math. J. {\bf 38} (1986), 57--69.

\bibitem{Ji} S. Ji: {\it Inequality for the distortion function of invertible
sheaves on abelian varieties,\/} Duke Math. J. {\bf 58} (1989), 657--667.

\bibitem{Kem} G.R. Kempf: {\it Metrics on invertible sheaves on abelian
varieties,\/} Topics in algebraic geometry (Guanajuato, 1989), pp.~107--108,
Aportaciones Mat. Notas Investigacion~5, Soc. Mat. Mexicana, Mexico, 1992.

\bibitem{Lue} H. Li, D. Luecking: {\it Schatten class of Hankel and Toeplitz
operators on the Bergman space of strongly pseudoconvex domains,\/}
Multivariable operator theory (Seattle, WA, 1993), pp.~237--257,
Contemp. Math.~185, Amer. Math. Soc., Providence, 1995.

\bibitem{Mgt} G. Mengotti: {\it The Bloch space for the minimal ball,\/}
Studia Math. {\bf 148} (2001), 131--142.

\bibitem{Mull} C. M\"uller, {\it Spherical harmonics,\/}
Lecture Notes Math.~17, Springer, Berlin, 1966.

\bibitem{MY} G. Mengotti, E.-H. Youssfi: {\it The weighted Bergman projection
and related theory on the minimal ball and applications,\/} Bull. Sci. Math.
{\bf 123} (1999), 501--525.

\bibitem{OPY} K. Oeljeklaus, P. Pflug, E.-H. Youssfi: {\it The Bergman kernel
of the minimal ball and applications,\/} Ann. Inst. Fourier (Grenoble)
{\bf 47} (1997), 915--928.

\bibitem{OY} K. Oeljeklaus, E.-H. Youssfi: {\it Proper holomorphic mappings
and related automorphism groups,\/} J.~Geom. Anal. {\bf 7} (1997), 623--636.

\bibitem{Rup} N. \O vrelid, J.~Ruppenthal: {\it $L^2$ properties of the $\dbar$
and $\dbar$-Neumann operator on spaces with isolated singularities,\/}
Math. Ann. {\bf 359} (2014), 803--838.

\bibitem{Rawn} J. Rawnsley: {\it Coherent states and K\"ahler manifolds,\/}
Quart. J. Math. Oxford~(2) {\bf 28} (1977), 403--415.

\bibitem{Ruan} W.D. Ruan: {\it Canonical coordinates and Bergman metrics,\/}
Comm. Anal. Geom. {\bf 6} (1998), 598--631.

\bibitem{Tian} G. Tian: {\it On a set of polarized K\"ahler metrics on
algebraic manifolds,\/} J.~Diff. Geom. {\bf 32} (1990), 99--130.

\bibitem{Wada} R. Wada: {\it On the Fourier-Borel transformation of analytic
functionals on the complex sphere,\/} T\^ohoku Math. J. {\bf 38} (1986),
417--432. 

\bibitem{Zeld} S. Zelditch: {\it Szeg\"o kernels and a theorem of Tian,\/}
Int. Math. Res. Not. {\bf 6} (1998), 317--331.

\bibitem{Zha} S. Zhang: {\it Heights and reductions of semi-stable
varieties,\/} Comp. Math. {\bf 104} (1996), 77--105.

\end{thebibliography}
\end{document}